\documentclass[11pt,a4paper]{article}
\usepackage[utf8]{inputenc}
\usepackage{amsmath}
\usepackage{amsfonts}
\usepackage{amssymb}
\usepackage{amsthm}
\usepackage{setspace}
\usepackage{hyperref}
\usepackage{graphicx}
\usepackage{stmaryrd}
\usepackage{mathtools}
\usepackage{subcaption}
\usepackage[T1]{fontenc}
\usepackage{xcolor}
\usepackage{float}
\usepackage{tikz-cd}
\usepackage[noadjust]{cite}
\usepackage{cite}
\usepackage{ragged2e}
\usepackage[margin=1.2in]{geometry}
\graphicspath{ {./images/} }
\newtheorem{theorem}{Theorem}[section]

\newtheorem{lemma}[theorem]{Lemma}
\newtheorem{corollary}[theorem]{Corollary}
\newtheorem{mainthm}{Theorem}

\theoremstyle{definition}
\newtheorem{definition}[theorem]{Definition}

\newtheorem{remark}[theorem]{Remark}

\newtheorem{cons}[theorem]{Construction}
\newtheorem{facts}[theorem]{Facts}

\DeclareMathOperator{\neutre}{\mathbf{e}}
\DeclareMathOperator{\N}{\mathbb{N}}
\DeclareMathOperator{\Coxgraph}{\widehat{\Gamma}}
\newcommand{\Coxlabel}{\widehat{m}}

\DeclareMathOperator{\Supp}{\operatorname{Supp}}

\newcommand{\Addresses}{{
		\bigskip
		\footnotesize
		
		Mireille Soergel, \textsc{Max Planck Institute for Mathematics in the Sciences, Inselstrasse 22, 04103 Leipzig, Germany}\par\nopagebreak
		\textit{E-mail address}: \texttt{soergel@mis.mpg.de}
        \medskip
        
            Nicolas Vaskou, \textsc{School of Mathematics, University of Bristol, Bristol BS8 1UG, UK}\par\nopagebreak
		\textit{E-mail address}: \texttt{nicolas.vaskou@gmail.com}
		
}}

\title{Dyer groups: Centres, hyperbolicity, and acylindrical hyperbolicity}
\author{Mireille Soergel and Nicolas Vaskou}

\begin{document}

\maketitle

\begin{abstract}
In this article we describe the centres of all Dyer groups. We also give a complete classification of when a Dyer group $D(\Gamma)$ is hyperbolic or acylindricality hyperbolic, with conditions that can easily be read on the Dyer graph $\Gamma$.
\end{abstract}

\section{Introduction}

Coxeter groups and right-angled Artin groups are generally well-understood. One common feature of Coxeter groups and right-angled Artin groups is their solution to the word problem. It was given by Tits for Coxeter groups \cite{Tits} and by Green for graph products of cyclic groups \cite{Green}. In his study of reflection subgroups of Coxeter groups \cite{Dyer1990ReflectionSubgroups}, Dyer introduces a family of groups which contains both Coxeter groups and graph products of cyclic groups. By \cite{Dyer1990ReflectionSubgroups} and \cite{ParSoe2023WordProblem}, this family, which we call Dyer groups, has the same solution to the word problem as Coxeter groups and graph products of cyclic groups. It is therefore natural to understand which properties of Coxeter groups and right-angled Artin groups can be extended to Dyer groups. A first answer can be found in \cite{Soergel2024Complex} where geometric actions of Dyer groups on CAT(0)
spaces are constructed that extend those of Coxeter groups on Davis–Moussong complexes (\cite{moussong1988hyperbolic}) and those of right-angled Artin groups on Salvetti complexes (\cite{ChaDav}). In the present work we study the centres of Dyer groups, and we focus on other aspects on non-positive curvature such as hyperbolicity and acylindrical hyperbolicity.

Similarly to Coxeter group and right-angled Artin groups, Dyer groups are defined by specific presentations which can be encoded in a labelled graph.
Let $\Gamma$ be a finite simplicial graph with vertex set $V(\Gamma)$ and edge set $E(\Gamma)$. We suppose that $\Gamma$ comes with a labelling $f : V(\Gamma) \rightarrow \N_{\geq 2}\cup\{\infty\}$ of its vertices and a labelling $m: E(\Gamma) \rightarrow \N_{\geq 2}$ of its edges, such that for any $\{u,v\}\in E(\Gamma)$, if $f(v) \geq 3$ then $m(u,v) = 2$. The graph $\Gamma$ is called a \emph{Dyer graph}, and the associated \emph{Dyer group} is the group $D(\Gamma)$ given by
\begin{multline*}
    D(\Gamma) \coloneqq \langle V(\Gamma) \mid v^{f(v)} = \neutre \text{ for all } v \in V(\Gamma) \text{ such that } f(v) < \infty, \\ \underbrace{uvu\cdots}_{m(u,v)\text{ terms}} = \underbrace{vuv\cdots}_{m(u,v) \text{ terms}} \text{ for all } \ \{u, v\} \in E(\Gamma) \rangle.
\end{multline*}

The reader unfamiliar with Coxeter groups or right-angled Artin groups may consider the following definition. A \emph{Coxeter group} is a Dyer group whose Dyer graph $\Gamma$ satisfies $f(v) = 2$ for every $v \in V(\Gamma)$. We will usually denote it by $W(\Gamma)$. A \emph{right-angled Artin group} is a Dyer group whose Dyer graph $\Gamma$ satisfies $f(v)=\infty$ for every $v\in V(\Gamma)$.
\bigskip

It is known (\cite{Dyer1990ReflectionSubgroups}) that for every subset $T \subseteq V(\Gamma)$, the subgroup $D_T$ generated by $T$ is isomorphic to the Dyer group $D(\Gamma')$, where $\Gamma'$ is the subgraph of $\Gamma$ spanned by $T$. Consequently, we will write interchangeably $D_T$ or $D(\Gamma')$ to denote the subgroup of $D(\Gamma)$ generated by $T = V(\Gamma')$. Such subgroups are called \emph{standard parabolic subgroups} of $D(\Gamma)$. 

A Dyer graph $\Gamma$ is said to be \emph{reducible} if $\Gamma$ can be decomposed as a join of two induced subgraphs, where every edge $e$ of the join satisfies $m(e) = 2$. Note that in this case, $D(\Gamma)$ splits as the direct product $D(\Gamma_1) \times D(\Gamma_2)$ of the Dyer groups given by the two subgraphs, and the Dyer group itself is called \emph{reducible}. The Dyer graph $\Gamma$ or the Dyer group $D(\Gamma)$ are called \emph{irreducible} otherwise. It is standard that $D(\Gamma)$ can always be decomposed as a direct product $D(\Gamma) = D(\Gamma_1) \times \cdots \times D(\Gamma_n)$ of irreducible standard parabolic subgroups, and that this decomposition is unique up to permuting the factors. The subgraphs $\Gamma_i$ are called the \emph{irreducible components} of $\Gamma$.

Our first result concern the centres of Dyer groups. Note that it is enough to deal with irreducible Dyer groups as the centre of a direct product is the direct product of the centres of the factors. Recall that the centres of Coxeter groups are fully classified (see Lemma \ref{LemmaCenterCoxeterGroup} for the compilation of results).

\begin{mainthm} \label{TheoremCentres}
    Let $D(\Gamma)$ be an irreducible Dyer group, and suppose that there exists some $v \in V(\Gamma)$ with $f(v) \neq 2$ (in other words, $\Gamma$ does not just define a Coxeter group). Then either $\Gamma$ is a single vertex and $D(\Gamma)$ is cyclic, or $D(\Gamma)$ has trivial centre.
\end{mainthm}

The notion of hyperbolic group was defined in the seminal work of Gromov \cite{gromov1987hyperbolic}, and has since then become a central notion in geometric group theory.

In the following, we give an easily checkable criterion regarding when a given Dyer group is Gromov-hyperbolic:

\begin{mainthm} \label{TheoremHyperbolicity}
    Let $D(\Gamma) \cong D(\Gamma_1) \times \cdots \times D(\Gamma_n)$ be a Dyer group decomposed as a product of its irreducible components. Then the following are equivalent:
    \begin{enumerate}
        \item $D(\Gamma)$ is hyperbolic ;
        \item At most one of the $D(\Gamma_i)$'s is infinite, and there is no standard parabolic subgroup $D_T$ of this $D(\Gamma_i)$ that is an affine Coxeter group of rank $\geq 3$ or that decomposes as $D_T = D_{T_1}\times D_{T_2}$ with $T = T_1\sqcup T_2$ and both $D_{T_1}$, $D_{T_2}$ infinite.
    \end{enumerate}
\end{mainthm}

There are several notions that generalise that of being Gromov-hyperbolic. In \cite{osin2016acylindrically}, Osin introduced the notion for a group to be acylindrically hyperbolic (see Definition \ref{DefiAH}). This broader class contains many groups that are not hyperbolic, such as $Out(F_n)$ ($n \geq 2$), mapping class groups, or many Artin groups, amongst others.

We also provide an easily verifiable condition for a Dyer group to be acylindrically hyperbolic:

\begin{mainthm} \label{TheoremAH}
    Let $D(\Gamma) \cong D(\Gamma_1) \times \cdots \times D(\Gamma_n)$ be a Dyer group decomposed as a product of its irreducible components. Then the following are equivalent:
    \begin{enumerate}
        \item $D(\Gamma)$ is acylindrically hyperbolic ;
        \item Exactly one of the $D(\Gamma_i)$'s is infinite, and it is not an affine Coxeter group and not isomorphic to $\mathbb{Z}$.
    \end{enumerate}
\end{mainthm}

These results were known for Coxeter groups and right-angled Artin groups. In Theorem \ref{ThmAH2}, we additionally give a classification of acylindrically hyperbolic Coxeter groups, which to our knowledge has not been written down in this form before, although the characterisation essentially follows from the work of \cite{caprace2010rank}.

\paragraph{Acknowledgements.} We thank the Centre de Recherches Math\'ematiques for its hospitality during the thematic semester on Geometric Group Theory and for providing a warm environnement for this project to begin. The second author is supported by the Postdoc Mobility $\sharp$P500PT$\_$210985 of the Swiss National Science Foundation.

\section{Centres of Dyer groups}

This section is devoted to the proof of Theorem \ref{TheoremCentres}. We start by introducing the necessary technical background.

\begin{definition}
    Let $S(V) = \{v^{\alpha}\mid v\in V, \alpha\in \mathbb{Z}_{f(v)}\setminus \{0\}\}$, where $\mathbb{Z}_{f(v)} = \mathbb{Z}/f(v)\mathbb{Z}$ if $f(v) <\infty$ and $\mathbb{Z}_{\infty} =\mathbb{Z}$. A syllabic word $w$ is an element in the free monoid $S(V)^* $. It is usually written as a finite sequence. For a syllabic word $w = (s_1,s_2,\dots, s_l)\in S(V)^*$, we set $\overline{w} = s_1s_2\cdots s_l\in D(\Gamma)$ and say that $\overline{w}$ is represented by $w$. Moreover $w$ is reduced if $l = \mathrm{lg}_{S(X)}(\overline{w})$.
    The support of a syllabic word $w = (v_1^{\alpha_1}, v_2^{\alpha_2}, \dots, v_l^{\alpha_l})$ is $\mathrm{Supp}(w) = \{v_1, v_2, \dots, v_l\}$. For $g\in D$, choose a reduced syllabic word $w = (v_1^{\alpha_1}, v_2^{\alpha_2}, \dots, v_l^{\alpha_l})$ representing $g$. We define $\mathrm{Supp}(g)= \mathrm{Supp}(w) = \{v_1, v_2, \dots, v_l\}$.
\end{definition}

\begin{remark}
    It was shown in \cite[Theorem 2.2]{ParSoe2023WordProblem} that the definition of $\Supp(g)$ does not depend on the choice of reduced word $w$ representing $g$.
\end{remark}

The following result is standard and describes the centres of Coxeter groups. It can essentially be found by putting together work from \cite{coxeter1935complete}, \cite{Bourbaki1968Groupes} and \cite{davis2008geometrytopologycoxeter}. It was rephrased in a unified way in \cite[Lemma 2.1]{Michael2024involutionscoxetergroups}:

\begin{lemma}\label{LemmaCenterCoxeterGroup}
    Let $W(\Gamma)$ be an irreducible Coxeter group. If $W(\Gamma)$ is of type $A_1$, $B_n$ ($n \geq 2$), $D_{2n}$ ($n \geq 2$), $E_7$, $E_8$, $G_2$, $F_4$, $H_3$, $H_4$ or $I_2(2m)$ ($m \geq 2$), then its centre is the subgroup $\langle w_0 \rangle$ of order $2$, where $w_0$ is the so-called longest element of $W(\Gamma)$. In all other cases the center of $W(\Gamma)$ is trivial.
\end{lemma}

The following is a direct consequence:

\begin{corollary} \label{CorollaryCenterCoxeterGroupSupport}
    Let $W(\Gamma)$ be an irreducible Coxeter group. If $z\in W(\Gamma)$ is central then either $z = \neutre$ or $\Supp(z) = V(\Gamma)$.
\end{corollary}

Before proving Theorem A, we recall the following lemma on centres of amalgamated products:

\begin{lemma} \cite{serre1977arbres} \label{lem: centres of amalgamated products}
    Let $G = A\ast_C B$ be the free product with amalgamation of two groups $A$ and $B$ along a common subgroup $C$. Then $Z(G)=Z(A)\cap Z(B) \leqslant C.$ 
\end{lemma}

\begin{remark} 
Note that a Dyer group $D(\Gamma)$ is (isomorphic to) a Coxeter group if and only if $f\equiv 2$ and it is a right-angled Artin group if and only if $f\equiv\infty$. One direction follows directly from the definition of Coxeter groups and right-angled Artin groups. The other direction follows from the abelianisation of the corresponding groups. Indeed the abelianisation of a Coxeter group (resp. a right-angled Artin group) is $(\mathbb{Z}_2)^k$ (resp. $\mathbb{Z}^k$) for some $k\in \mathbb{N}$ and the abelianisation of a Dyer group is $\prod_{p\in im(f)}\mathbb{Z}_p^{k_p}$, where $1\leq k_p\leq |\{v\in V(\Gamma)\mid f(v) = p\}|$. 
\end{remark}

\begin{theorem}[Theorem A]
    Let $D(\Gamma)$ be an irreducible Dyer group, and suppose that there exists some $v \in V(\Gamma)$ with $f(v) \neq 2$ (in other words, $\Gamma$ does not just define a Coxeter group). Then either $\Gamma$ is a single vertex and $D(\Gamma)$ is cyclic, or $D(\Gamma)$ has trivial centre.
\end{theorem}

\begin{proof}[Proof of Theorem \ref{TheoremCentres}]
Let $V = V(\Gamma)$. By hypothesis, there is a vertex $v \in V$ such that $f(v) \neq 2$. Note that $m(v, u) = 2$ for every $u \in V$ adjacent to $v$, so by irreducibility of $\Gamma$, there must be a vertex $u \in V$ that is not adjacent to $v$. We can thus write $D(\Gamma)$ as an amalgamated free product
\[D(\Gamma) \cong D_{V\setminus\{u\}}\ast_{D_{V\setminus\{u, v\}}} D_{V\setminus\{v\}}.\]

\medskip \noindent By Lemma \ref{lem: centres of amalgamated products}, this implies that $Z(D(\Gamma)) \leqslant D_{V\setminus\{u, v\}}$. Note that this inclusion holds for any pair of non-adjacent vertices $u, v$ in $\Gamma$. In particular, we obtain that $Z(D(\Gamma)) \leqslant D_C$ where $C= \{v\in V \mid \mathrm{Star}(v) = \Gamma\} \subsetneq V(\Gamma)$. Note that for every $v\in C$ we must have $f(v) = 2$ by irredubility of $\Gamma$. Consequently, the parabolic subgroup $D_C$ is a Coxeter group.

Let $\Gamma_C$ be the subgraph of $\Gamma$ spanned by $C$, and let $C_1, \dots, C_k \subseteq C$ be such that $\Gamma_{C_1}, \dots, \Gamma_{C_k}$ are the irreducible components of $\Gamma_C$. For every $i \in \{1, \dots, k\}$, there exists some $v_i \in C_i$ and $u_i\in V\setminus C_i$ such that $m(v_i, u_i) \neq 2$, or $\Gamma_{C_i}$ would be a proper irreducible component of $\Gamma$, contradicting irreducibility. Note that in particular, $f(u_i) = 2$ for all $i \in \{1, \dots, k\}$. As $\Gamma_{C_1}, \dots, \Gamma_{C_k}$ are irreducible components of $\Gamma_{C}$, any edge connecting two vertices of distinct $\Gamma_{C_i}, \Gamma_{C_j}$ has label $2$. In particular, it must be that $u_i \notin C$ for all $i \in \{1, \dots, k\}$ (see Figure \ref{FigureCentres}). Let now $C^+=C\sqcup\{u_1,\dots,u_k\}$, and let $\Gamma_{C^+}$ denote the subgraph of $\Gamma$ spanned by $C^+$. By the above, the parabolic subgroup $D_{C^+}$ is also a Coxeter group.

One notices that we have $Z(D(\Gamma)) \leqslant D_C \leqslant D_{C^+}$, and thus we also have $Z(D) \leqslant Z(D_C) \cap Z(D_{C^+})$. Let now $z \in Z(D(\Gamma))$, and suppose that $z \neq \neutre$. As $D_C$ and $D_{C^+}$ are both Coxeter groups and $z$ belongs to both centres, we can apply Corollary \ref{CorollaryCenterCoxeterGroupSupport} twice. On one hand, $\Supp(z)$ is (the vertex set of) a non-trivial union of irreducible components of $C$. On the other hand, $\Supp(z)$ is (the vertex set of) a non-trivial union of irreducible components of $C^+$. But each irreducible component of $C^+$ contains one of the $u_i$, none of which are contained in $C$. This gives a contradiction, which proves that $Z(D(\Gamma))$ is trivial.
\end{proof}

\begin{figure}[H]
\centering
\includegraphics[scale=1]{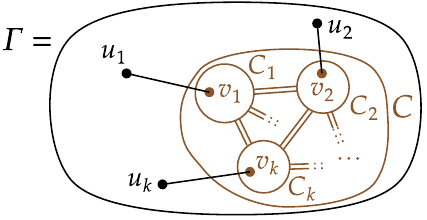}
\caption{The double edges represent joins between the components $C_i$. Each edge of the join has label $2$. All vertices on the picture have label $2$. The single black edges have label at least $3$.}

\label{FigureCentres}
\end{figure}

\section{Hyperbolicity}

In this section we prove Theorem \ref{TheoremHyperbolicity}.

\begin{definition} 
    Let $D(\Gamma)$ be a Dyer group, and let $V_2 = \{v\in V(\Gamma) \mid f(v) = 2\}$, $V_{\infty} = \{v\in V(\Gamma) \mid f(v) = \infty\}$ and $V_p = V(\Gamma) \setminus\{V_2\cup V_{\infty}\}$.
\end{definition}

\begin{cons} \label{ConstructionCoxeterGraph} Let $D(\Gamma)$ be a Dyer group. We construct a new Dyer graph $\Coxgraph = (\Coxgraph, \Coxlabel)$ that defines a Coxeter group $W(\Coxgraph)$ as follows:
\begin{itemize}
    \item Its vertex set is the disjoint union $V(\Coxgraph) \coloneqq V(\Gamma) \sqcup (V_p \cup V_{\infty})$. If $v$ is a vertex of $V_p \cup V_{\infty}$, we will denote by $v'$ the corresponding vertex in the disjoint copy $V_p \cup V_{\infty}$ ;
    \item Two vertices $u, v \in V(\Gamma)$ span an edge of $\Coxgraph$ if and only if they span an edge of $\Gamma$, and we set $\Coxlabel(u, v) \coloneqq m(u, v)$ ;
    \item For all $v \in V_p \cup V_{\infty}$ and $u \in V(\Coxgraph) \setminus \{v, v'\}$ we set $\Coxlabel(v', u) \coloneqq 2$ ;
    \item For all $v \in V_p$, there is an edge $\{v, v'\}$ in $\Coxgraph$ labelled by $\Coxlabel(v, v') \coloneqq f(v)$.
\end{itemize}
\end{cons}

\begin{figure}[H]
\centering
\includegraphics[scale=0.88]{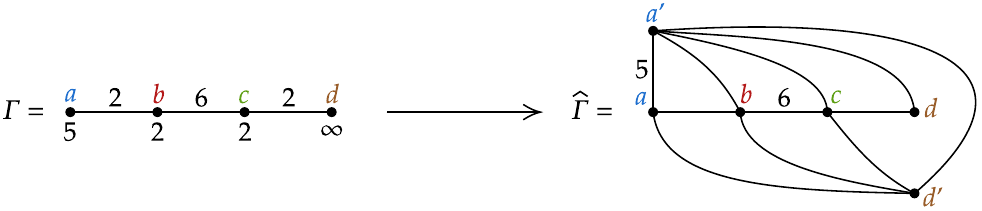}
\caption{An example of a Dyer graph $\Gamma$ and the associated Dyer graph $\Coxgraph$ corresponding to the Coxeter group $W(\Coxgraph)$. The vertices and edges whose label are not specified have label $2$.}
\label{FigureExample}
\end{figure}

The following immediately follows from Construction \ref{ConstructionCoxeterGraph}:

\begin{corollary} \label{CorollaryNumberOfNon2Edges}
    Let $D(\Gamma)$ be a Dyer group. If $v \in V_{\infty}$, then all the edges adjacent to $v$ or to $v'$ in $\Coxgraph$ have label $2$. If $v \in V_p$, then the only edge adjacent to $v$ or to $v'$ in $\Coxgraph$ that does not have label $2$ is the edge $\{v, v'\}$.
\end{corollary}

\begin{theorem}\cite[Theorem 2.8]{Soergel2024Complex} \label{ThmFiniteIndexEmbedding}
    Let $D(\Gamma)$ be a Dyer group, and consider the associated Coxeter group $W(\Coxgraph)$ from Construction \ref{ConstructionCoxeterGraph}. Then $W(\Coxgraph) \cong D(\Gamma) \rtimes (\mathbb{Z}/2\mathbb{Z})^{|V_p\cup V_{\infty}|}$. In particular, $D(\Gamma)$ has finite index in $W(\Coxgraph)$.
\end{theorem}

\begin{lemma} \label{LemmaIrreducibilityPassesToW}
    Let $\Gamma$ be a Dyer graph with irreducible components $\Gamma_1, \dots, \Gamma_n$. Then the irreducible components of $\Coxgraph$ are $\Coxgraph_1, \dots, \Coxgraph_n$.
\end{lemma}

\begin{proof}
Let us consider any $i \in \{1, \dots, n\}$. On one hand, $\Gamma_i$ appears as a subgraph of $\Coxgraph_i$, and it is irreducible by hypothesis. On the other hand, for every $v \in \Gamma_i \cap (V_p \cup V_{\infty})$ we have $\Coxlabel(v, v') = f(v) \neq 2$. In particular, $v$ and $v'$ are in the same irreducible component of $\Coxgraph$. This shows that $\Coxgraph_i$ is irreducible. Now let $v \in V(\Coxgraph_i)$ be any vertex, and let $u \in V(\Coxgraph) \setminus V(\Coxgraph_i)$. Then it follows from Construction \ref{ConstructionCoxeterGraph} that $\Coxlabel(v, u) = 2$. Consequently, $\Coxgraph_i$ is an irreducible component of $\Coxgraph$.
\end{proof}

Let $W(\Gamma)$ be a Coxeter group. The \emph{Coxeter graph} associated with $W(\Gamma)$ is the graph $\Gamma_{Cox}$ whose vertex set if the same as that of $\Gamma$, but between two vertices $u, v$ we draw:
\begin{itemize}
\item No edge if $m(u, v) = 2$ ;
\item An unlabelled edge if $m(u, v) = 3$ ;
\item An edge labelled $m(u, v)$ if $m(u, v) \in \{4, 5, \dots\} \cup \{\infty\}$.
\end{itemize}

\begin{definition} \cite{Humphreys1990ReflectionGroups} \label{DefClassificationAffineCoxeter}
An irreducible Coxeter group $W(\Gamma)$ is called \emph{affine} if its Coxeter graph $\Gamma_{Cox}$ belongs to the list described on Figure \ref{FigureAffines}.

A Coxeter group is called \emph{affine} if its irreducible components are all affine or finite, and at least one is affine.

\begin{figure}[H]
\centering
\includegraphics[scale=0.8]{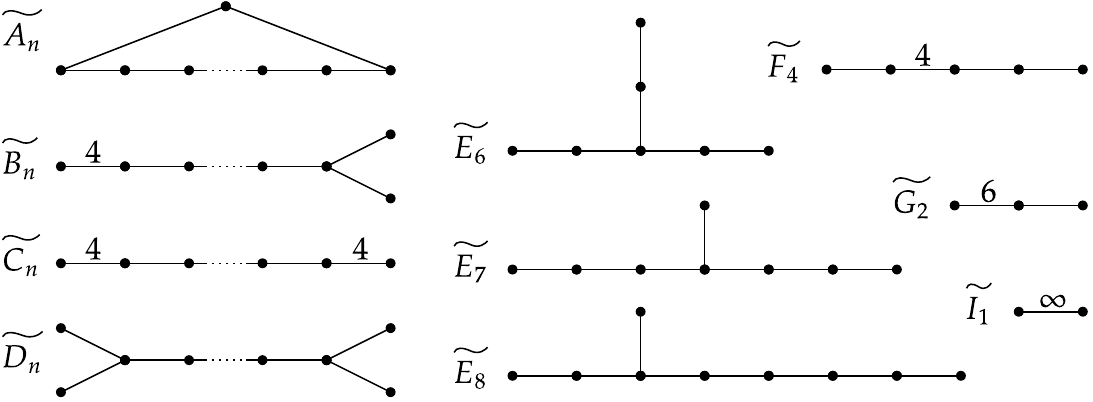}
\caption{The irreducible affine Coxeter groups. The graphs for $\widetilde{A_n}$ ($n \geq 2$), $\widetilde{B_n}$ ($n \geq 3$), $\widetilde{C_n}$ ($n \geq 2$) and $\widetilde{D_n}$ ($n \geq 4$) have $n+1$ vertices.}
\label{FigureAffines}
\end{figure}
\end{definition}

We recall that the rank of a Coxeter group $W(\Gamma)$ is $|V(\Gamma)|$.

\begin{theorem}  [Theorem B]
    Let $D(\Gamma)$ be a Dyer group, and let $\Gamma_1, \dots, \Gamma_n$ be the irreducible components of $\Gamma$. Then the following are equivalent:
    \begin{enumerate}
        \item $D(\Gamma)$ is hyperbolic ;
        \item At most one of the $D(\Gamma_i)$'s is infinite, and there is no standard parabolic subgroup $D_T$ of this $D(\Gamma_i)$ that is an affine Coxeter group of rank $\geq 3$ or that decomposes as $D_T = D_{T_1}\times D_{T_2}$ with $T = T_1\sqcup T_2$ and both $D_{T_1}$, $D_{T_2}$ infinite ;
        \item $W(\Coxgraph)$ is hyperbolic ;
        \item At most one of the $W(\Coxgraph_i)$'s is infinite, and there is no standard parabolic subgroup $W_T$ of this $W(\Coxgraph_i)$ that is an affine Coxeter group of rank $\geq 3$ or that decomposes as $W_T = W_{T_1}\times W_{T_2}$ with $T = T_1\sqcup T_2$ and both $D_{T_1}$, $D_{T_2}$ infinite.
    \end{enumerate}
\end{theorem}

\begin{proof}
    \noindent [(3) $\Leftrightarrow$ (4)] This is a standard result, see \cite[Theorem 17.1]{moussong1988hyperbolic}.
    \medskip

    \noindent [(1) $\Leftrightarrow$ (3)] This follows directly from Theorem \ref{ThmFiniteIndexEmbedding}.
    \medskip
    
    \noindent [(2) $\Leftrightarrow$ (4)] Combining Theorem \ref{ThmFiniteIndexEmbedding} and Lemma \ref{LemmaIrreducibilityPassesToW} shows that an irreducible component $D(\Gamma_i)$ is infinite if and only if the corresponding irreducible component $W(\Coxgraph_i)$ is infinite. So $D(\Gamma)$ has at most one infinite irreducible factor if and only if $W(\Coxgraph)$ has at most one infinite irreducible factor. All that's left to show is the equivalence on the existence of certain standard parabolic subgroups of $D(\Gamma_i)$ and of $W(\Coxgraph_i)$, and for this we may assume that $n = 1$, i.e. that $\Gamma = \Gamma_i$ and $\Coxgraph = \Coxgraph_i$.
    \medskip

    Suppose that there is a subset $T \subseteq V(\Gamma)$ such that $D_T \leqslant D(\Gamma)$ decomposes as $D_T = D_{T_1} \times D_{T_2}$, with $T = T_1 \sqcup T_2$ and both $D_{T_1}$, $D_{T_2}$ infinite. Then for $\widehat{T} \coloneqq T\cup \{v' \mid v\in T\cap (V_p\cup V_{\infty})\} \subset V(\Coxgraph)$, the group $D_T$ is a finite index subgroup of $W_{\widehat{T}}$. By Lemma \ref{LemmaIrreducibilityPassesToW} we have $W_{\widehat{T}} = W_{\widehat{T_1}}\times W_{\widehat{T_2}}$, with $\widehat{T} = \widehat{T_1}\sqcup\widehat{T_2}$ and both $W_{\widehat{T_1}}$,  $W_{\widehat{T_2}}$ infinite (note that $D_{T_1}$ and $D_{T_2}$ are products of irreducible components). 
    
    Suppose now that there is a subset $T\subseteq V(\Coxgraph)$ such that $W_T \leqslant W(\Coxgraph)$ decomposes as $W_T = W_{T_1}\times W_{T_2}$, with $T = T_1\sqcup T_2$ and both $W_{T_1}$, $W_{T_2}$ infinite. If for all $v\in V_p\cup V_{\infty}$ we have that $v\in T$ if and only if $v'\in T$, then one can write $T = \widehat{T\cap V(\Gamma)}$. Using Lemma \ref{LemmaIrreducibilityPassesToW}, we have $D_{T\cap V(\Gamma)} = D_{T_1\cap V(\Gamma)}\times D_{T_2\cap V(\Gamma)}$ where $D_{T_1\cap V(\Gamma)}$, $D_{T_2\cap V(\Gamma)}$ are infinite. Otherwise, consider
    $$\widetilde{T} \coloneqq (T\setminus \{v'\mid v'\in T \text{ and } v\notin T\} )\cup \{v'\mid v\in T\cap (V_p\cup V_{\infty}) \}.$$
    Define $\widetilde{T_1}$ and $\widetilde{T_2}$ similarly. We need to show that $W_{\widetilde{T}}$ decomposes as a direct product of $W_{\widetilde{T_1}}$ and $W_{\widetilde{T_2}}$, both infinite. Note that by Construction \ref{ConstructionCoxeterGraph}, we have
    $$W_T = W_{T\setminus \{v'\mid v'\in T \text{ and } v\notin T\}}\times W_{\{v'\mid v'\in T \text{ and } v\notin T\}},$$ where $W_{\{v'\mid v'\in T \text{ and } v\notin T\}}$ is finite. Moreover
    $W_{\{v'\mid v\in T_i\cap (V_p\cup V_{\infty}) \}}$ commutes with $W_{T_j}$ and with $W_{\{v'\mid v\in T_j\cap (V_p\cup V_{\infty}) \}}$ for $\{i, j\} = \{1, 2\}$. Hence $W_{\widetilde{T}} = W_{\widetilde{T_1}}\times W_{\widetilde{T_2}}$ where $\widetilde{T} = \widetilde{T_1}\sqcup \widetilde{T_2}$ and both $W_{\widetilde{T_1}}$, $W_{\widetilde{T_2}}$ are infinite, and for all $v\in V_p \cup V_{\infty}$ we have $v\in\widetilde{T}$ if and only if $v'\in\widetilde{T}$. So we can conlude that there is some $T\subset V(\Gamma)$ such that $D_T = D_{T_1}\times D_{T_2}$, with $T = T_1\sqcup T_2$ and both $D_{T_1}$, $D_{T_2}$ infinite if and only if there is some $T\subset V(\Coxgraph)$ such that $W_T = W_{T_1}\times W_{T_2}$, with $T = T_1\cup T_2$ and both $W_{T_1}$, $W_{T_2}$ infinite.
    \medskip

    Suppose that there is a standard parabolic subgroup $W_T \leq W(\Coxgraph)$ which is an affine Coxeter group of rank $\geq 3$. We can assume that $W_T$ is irreducible. So $\Coxgraph_T$ is one of the graphs of Figure \ref{FigureAffines}. Consider the graphs in Figure \ref{FigureAffines}. None of their vertices satisfy the conditions of Corollary \ref{CorollaryNumberOfNon2Edges}. So one can see that $T\cap \{v, v'\mid v\in V_p\cup V_{\infty}\} = \emptyset$, as none of the vertices of $T$ satisfy the conditions of Corollary \ref{CorollaryNumberOfNon2Edges}.
    So $T \subseteq V_2 \subseteq V(\Gamma)$ and so $D_T \leqslant D(\Gamma)$ is an affine Coxeter group of rank $\geq 3$.
    
    On the other hand if a standard parabolic subgroup $D_T \leqslant D(\Gamma)$ is an affine Coxeter group of rank $\geq 3$ then $T \subseteq V_2$. By construction $T \subseteq V(\Coxgraph)$ and $W_T$ is also an affine Coxeter group of rank $\geq 3$. We conclude that there is a $T \subseteq V(\Coxgraph)$ such that $W_T$ is an affine Coxeter group of rank $\geq 3$ if and only if there is a $T \subseteq V(\Gamma)$ such $D_T$ is an affine Coxeter group of rank $\geq 3$.
\end{proof}

\section{Acylindrical hyperbolicity}

Finally, in this section we prove Theorem \ref{TheoremAH}. We start by recalling the definition of acylindrically hyperbolic groups:

\begin{definition} \label{DefiAH}
    Let $G$ be a group acting by isometries on a metric space $(X, d)$. We say that the action is \emph{acylindrical} if for every $R \geq 0$, there is some $N \geq 0$ and $L \geq 0$ such that for every $x, y \in X$, if $d(x, y) \geq L$ then
    $$\# \{g \in G \ | \ d(x, gx) \leq R, \ d(y, gy) \leq R \} \leq N.$$
    A group $G$ is said to be \emph{acylindrically hyperbolic} if it is not virtually cyclic and it admits an acylindrical action on a hyperbolic space $(X, d)$.
\end{definition}

In the following, we compile several well-known facts about acylindrically hyperbolic groups:

\begin{facts} \label{FactsAH} \cite{osin2016acylindrically}
    (1) Acylindrical hyperbolicity passes to finite index subgroups ;
    \\(2) Acylindrically hyperbolic groups do not split as direct products of infinite groups ;
    \\(3) Acylindrically hyperbolic groups always contain non-abelian free subgroups.
\end{facts}

We first prove the following:

\begin{lemma} \label{LemmaAffineIIFAffine}
    Let $D(\Gamma)$ an irreducible Dyer group. Then the following are equivalent:
    \begin{itemize}
        \item $D(\Gamma)$ is an affine Coxeter group of rank $\geq 3$.
        \item $W(\Coxgraph)$ is affine of rank $\geq 3$.
    \end{itemize}
    Moreover $W(\Coxgraph)$ is of type $\widetilde{I_1}$ if and only if $D(\Gamma)$ is $\mathbb{Z}$ or a Coxeter group of type $\widetilde{I_1}$.
\end{lemma}

\begin{proof}
If $D(\Gamma)$ is an affine Coxeter group of rank $\geq 3$, we have $\Coxgraph = \Gamma$. In particular, $D(\Gamma) = W(\Coxgraph)$ so the statement is tautological. We focus on the other direction.

Let $v \in V_p \cup V_{\infty}$. By Corollary \ref{CorollaryNumberOfNon2Edges}, the two vertices $v, v' \in V(\Coxgraph)$ are not joined by an edge labelled $2$. Moreover, every edge of $\Coxgraph$ of the form $\{v, u\}$ or $\{v', u\}$ where $u \in V(\Coxgraph) \setminus \{v, v'\}$ is labelled $2$, so every edge of $\Coxgraph_{Cox}$ of the form $\{v, u\}$ or $\{v', u\}$ where $u \in V(\Coxgraph) \setminus \{v, v'\}$ is labelled $\infty$. By Definition \ref{DefClassificationAffineCoxeter}, this directly implies that either $W(\Coxgraph)$ is not affine, or that it is affine of rank $2$ (and $\Coxgraph$ is of type $\widetilde{I_1}$). This gives a contradiction, so it must be that $V_p \cup V_{\infty} = \emptyset$. But then it follows that $\Coxgraph = \Gamma$ and $D(\Gamma) = W(\Coxgraph)$ is an affine Coxeter group of rank $\geq 3$.

The last statement of the lemma is clear.
\end{proof}

\begin{theorem} [Theorem C] \label{ThmAH2}
    Let $D(\Gamma)$ be a Dyer group, and let $\Gamma_1, \dots, \Gamma_n$ be the irreducible components of $\Gamma$. Then the following are equivalent:
    \begin{enumerate}
        \item $D(\Gamma)$ is acylindrically hyperbolic ;
        \item Exactly one of the $D(\Gamma_i)$'s is infinite, and it is not an affine Coxeter group and not isomorphic to $\mathbb{Z}$ ;
        \item $W(\Coxgraph)$ is acylindrically hyperbolic ;
        \item Exactly one of the $W(\Coxgraph_i)$'s is infinite, and it is not affine.
    \end{enumerate}
\end{theorem}

\begin{proof}
    \noindent [(4) $\Rightarrow$ (3)] This is a standard result, but we recall the argument here. It is known from the work of Moussong \cite{moussong1988hyperbolic} that any Coxeter group $W(\Lambda)$ acts geometrically on its associated Davis complex $\Sigma_{\Lambda}$ and that this complex is CAT(0) for some piecewise Euclidean metric. Decompose $W(\Lambda) = W(\Lambda_1) \times \dots \times W(\Lambda_k)$ as a product of its irreducible factors. Then it was proved in \cite[Proposition 4.5]{caprace2010rank} that $W(\Lambda)$ contains a rank-one element for the action on $\Sigma_{\Lambda}$ if and only if exactly one of the $W(\Lambda_i)$ is infinite, and it is not an affine Coxeter group of rank $\geq 3$. In particular then, our $W(\Coxgraph)$ admits a rank-one element for its action on $\Sigma_{\Coxgraph}$.

    Sisto (\cite[Theorem 1.5]{sisto2018contracting}) proved that for proper isometric actions on proper CAT(0) spaces (such as $W(\Coxgraph) \curvearrowright \Sigma_{\Coxgraph}$), every element acting as a rank-one isometry is contained in an hyperbolically embedded virtually cyclic subgroup. In particular, it then follows from \cite[Theorem 1.2]{osin2016acylindrically} that either $W(\Coxgraph)$ is virtually cyclic, or it is acylindrically hyperbolic. But $W(\Coxgraph)$ is infinite, so it can't be virtually cyclic, or it would be affine. 
    \medskip

    \noindent [(3) $\Rightarrow$ (1)] This directly follows from Theorem \ref{ThmFiniteIndexEmbedding} and the fact that acylindrical hyperbolicity passes to finite index subgroups (see Fact \ref{FactsAH}.(1)).
    \medskip
    
    \noindent [(1) $\Rightarrow$ (2)] By Fact \ref{FactsAH}.(2) we cannot have two distinct infinite factors $D(\Gamma_i)$ and $D(\Gamma_j)$. If no factor was infinite, then $D(\Gamma)$ would be finite hence also not acylindrically hyperbolic. So we assume that all factors are finite except one factor $D(\Gamma_1)$. If $D(\Gamma_1)$ is an irreducible affine Coxeter group, then $D(\Gamma_1)$ (and thus $D(\Gamma)$) is virtually abelian. This also contradicts acylindrical hyperbolicity (see Fact \ref{FactsAH}.(3)). The same argument holds if $D(\Gamma_1) \cong \mathbb{Z}$.
    \medskip

    \noindent [(2) $\Rightarrow$ (4)] The fact that exactly one irreducible factor $W(\Coxgraph_i)$ of $W(\Coxgraph)$ is infinite is a consequence of Lemma \ref{LemmaIrreducibilityPassesToW} and Theorem \ref{ThmFiniteIndexEmbedding}. Moreover it can't be that $W(\Coxgraph_i)$ is affine of rank $\geq 3$, or $D(\Gamma_i)$ would be an affine Coxeter group of rank $\geq 3$ by Lemma \ref{LemmaAffineIIFAffine}.
    
    Finally, if $W(\Coxgraph_i)$ were affine of rank $\leq 2$ then it has to be $\widetilde{I_1}$ by Definition \ref{DefClassificationAffineCoxeter}. By Lemma \ref{LemmaAffineIIFAffine}, this means that $D(\Gamma)$ is either a Coxeter group of type $\widetilde{I_1}$, or it is isomorphic to $\mathbb{Z}$, contradicting our assumptions.
\end{proof}

\bibliographystyle{alpha}
\bibliography{Mainbib}

\Addresses

\end{document}